\documentclass[12pt]{article}

\usepackage{fouriernc, enumerate, graphicx, multicol, fullpage, setspace, amssymb, amsmath, tikz-cd,amsthm,wrapfig}
\usepackage{enumitem}
\usepackage{titlesec}
\usepackage[colorlinks=true,urlcolor=blue,pdftex,
            pdftitle={pdftitle},
            pdflang={en-US}]{hyperref}
\usepackage[total={6.5in,9in}, top=1in, left=1in]{geometry}
\usepackage{alphalph}

\usepackage[all]{xy}

\DeclareMathAlphabet{\mathcal}{OMS}{cmsy}{m}{n}
\theoremstyle{definition} 
\usepackage[style=alphabetic]{biblatex}%bibliography stuff
\newcommand{\PP}{\mathbb{P}}

\renewcommand{\to}{\longrightarrow}
\DeclareMathOperator{\length}{length}
\renewcommand{\span}{\operatorname{span}}
\DeclareMathOperator{\EJ}{EJ}
\DeclareMathOperator{\RJ}{RJ}
\DeclareMathOperator{\Sec}{Sec}

\newtheorem{theorem}{Theorem}
\newtheorem{corollary}[theorem]{Corollary}
\newtheorem{lemma}[theorem]{Lemma}
\newtheorem{proposition}[theorem]{Proposition}
\newtheorem*{definition}{Definition}

\newtheorem*{remark}{Remark}

\newtheorem*{example}{Example}

\title{Generalizing M. Dale's results from secants to joins}
\author{J. Beckmann}

\begin{document}
\maketitle
\thispagestyle{empty}
%\vspace{.3in}
\noindent Abstract: Magnar Dale's paper ``Terracini's lemma and the secant variety of a curve" contains various facts about secant varieties, nearly all of whose proofs can immediately be extended to the situation of embedded joins of varieties. This note provides the necessary details on how to do so, and as an application shows how to use this information to calculate the degree of the canonical map from the ruled join down to the embedded join.
\section*{Introduction}
We are not the first to extend Dale's results to embedded joins; \AA dlandsvik (\cite{dlandsvik1987JoinsAH}) extended Dale's version of Terracini's lemma to joins, and Ballico (\cite{EBallico}) partially extended one of Dale's other results to joins. The primary reason why we wish to extend the rest of Dale's results is their application in explicit calculations using the refined B\'ezout theorem (depicted on the cover of \cite{flenner2013joins}, see inside that book for details). This theorem requires calculating the degree of the map $\pi:\RJ(X,Y)\to \EJ(X,Y)$, the canonical projection of the ruled join variety $\RJ(X,Y)\subseteq\PP^{2n+1}$ to the embedded join variety $\EJ(X,Y)\subseteq\PP^n$ (each of which is defined in \S1). We can phrase some of the results of our main theorem (Theorem \ref{main}) as 
\begin{theorem}
    Let $X,Y\subset\PP^n$ be closed subschemes, not a constrained pair, and assume $\dim(\EJ(X,Y))=\dim(X)+\dim(Y)+1$. Let $L$ be a general line in the ruling of $\EJ(X,Y)$, and let $m_X:=\length(X\cap L)$ and $m_Y:=\length(Y\cap L)$. Let $z\in L$ be a general point, and let $b$ equal the total number of lines in the ruling of $\EJ(X,Y)$ passing through $z$. Then
    $$\deg(\pi)=m_Xm_Yb$$
\end{theorem}
\noindent (Constrained pairs are defined in \S1.) The numbers $m_X$, $m_Y$, and $b$ correspond to the degrees of certain maps $\alpha_X$, $\alpha_Y$, and $\beta$, respectively, which are based on the maps $\alpha$ and $\beta$ used by Dale.

In \S1 we rebuild the constructions from Dale's paper (\S1 of \cite{Dale}) with appropriate modifications for our current purposes, and define the new terminology we need to properly state the results of \S\S2 and 3. In \S2 we state, prove, and discuss our generalizations of Dale's results. In \S3 we prove our motivating fact, Theorem \ref{main}. In \S4, we outline some remaining questions.

\begin{remark} This note is not self-contained, and was written to be read alongside \cite{Dale}. We make ubiquitous reference to how things were accomplished in \cite{Dale}, we base much of our notation on that of \cite{Dale}, and have elected to exclude certain proofs, as they are identical to Dale's after swapping appropriate notation.
\end{remark}

\section{Notation}
\subsection{Basic maps and correspondences}
Assume $k$ is algebraically closed. Let $X,Y\subseteq\PP^n$ be closed, reduced, irreducible subschemes, and exclude the cases where $X=Y$ is a linear space. The embedded join $\EJ(X,Y)\subseteq\PP^n$ is the closure of the union of all lines $\overline{x,y}$ for $x\in X$, $y\in Y$, and $x\neq y$. Let $\PP^{2n+1}$ have homogeneous coordinates $[a_0,\dots,a_n,b_0,\dots,b_n]$, and let $X'\subseteq V(a_0,\dots,a_n)\subset\PP^{2n+1}\supset V(b_0,\dots,b_n)\supseteq Y'$ be two isomorphically embedded copies of $X$ and $Y$, embedded in the indicated complementary $n$-planes. We call $\RJ(X,Y):=\EJ(X',Y')$ the \textit{ruled join} of $X$ and $Y$. The map $\begin{tikzcd}pr:\PP^{2n+1}\arrow[r,dashed]&\PP^n\end{tikzcd}$ given by projection from the $n$-plane $V(a_0-b_0,\dots,a_n-b_n)$ restricts to the canonical map $\begin{tikzcd}\pi:\RJ(X,Y)\arrow[r,dashed]&\EJ(X,Y)\end{tikzcd}$.

Let $G=G(1,n)$, the Grassmannian of lines in $\PP^n$. Denote by $\Gamma(X,Y)\subset\PP^n\times\PP^n\times G$ the closure of the graph of the morphism $\begin{tikzcd}s:X\times Y\arrow[r,dashed]&G\end{tikzcd}$, given by $s(x,y)=\overline{x,y}$. $\Gamma(X,Y)$ is the blow up of $X\times Y$ at $\Delta\cap (X\times Y)$, where $\Delta\subset\PP^n\times\PP^n$ is the diagonal. $\Gamma(X,Y)$ is irreducible and its dimensions is $\dim(X)+\dim(Y)$.

Put $JL(X,Y)=pr_3(\Gamma(X,Y))\subset G$; this plays the role of Dale's $S(X)$. The general fiber of $pr_3|_{\Gamma(X,Y)}=\overline{pr}_3:\Gamma(X,Y)\to JL(X,Y)$ is finite except in the case that $X=Y$ is a linear space, and so $\dim(JL(X,Y))=\dim(X)+\dim(Y)$ and $JL(X,Y)$ is irreducible. Then $\displaystyle{\EJ(X,Y)=\bigcup_{L\in JL(X,Y)}L}$, and we call lines in $JL(X,Y)$ \textit{join lines} of $\EJ(X,Y)$.

Let $I=\{(p,l):p\in l\}\subset\PP^n\times G$ be the incidence correspondence, and put $JB(X,Y)=(\overline{pr}_2)^{-1}(JL(X,Y))\subseteq I$, where $\overline{pr}_2:I\to G$. This will play of role of Dale's $SB(X)$. If we give $JL(X,Y)$ the reduced scheme structure, then $JB(X,Y)$ is reduced and irreducible of dimension $\dim(X)+\dim(Y)+1$. Then $\EJ(X,Y)=pr_1(JB(X,Y))$, and so $\EJ(X,Y)$ is irreducible of dimension at most $\dim(X)+\dim(Y)+1$.

We must now diverge more clearly from Dale. Consider the diagram
$$\begin{tikzcd}
    \mathrm{Bl}_{X\cap Y}(X\times Y)\arrow[rr,"\cong"]\arrow[d,hook]&&\arrow[d,hook]\Gamma(X,Y)\\
    \arrow[dd,"\lambda"]
    \mathrm{Bl}_\Delta(\PP^n\times\PP^n)
    \arrow[rr,"\cong"]\arrow[dr,"\pi"]&&
    \arrow[dd,"\overline{pr}_{23}"]\arrow[ddrd,bend left=20,"\overline{pr}_{13}"]
    \Gamma(\PP^n,\PP^n)\arrow[dl,"\overline{pr}_{12}"]\arrow[r, phantom, sloped, "\subset"]&\PP^n\times\PP^n\times G\\
    &\PP^n\times\PP^n&\\
    \PP(\Omega_{\PP^n}^1)\arrow[rrrd,bend right=10,"\cong"]\arrow[rr,"\cong"]&&I\\
    &&&I
\end{tikzcd}$$
where $\pi$ is the blowdown map and $\lambda$ is the canonical projection of ($\mathrm{Bl}_\Delta(\PP^n\times\PP^n)$ as a bundle over $\PP(\Omega_{\PP^n}^1)$). Because our first and second components of $\Gamma(X,Y)$ are not assumed to be equal as in the case of secants, we need to consider both
$$N(X)=\overline{pr}^{-1}_{13}(pr_{13}(\Gamma(X,Y))\hspace{.5in}\mathrm{and}\hspace{.5in}N(Y)=\overline{pr}^{-1}_{23}(pr_{23}(\Gamma(X,Y))$$
as separate components. These will partially play the role of Dale's $M(X)$. Set theoretically, $N(X)=\{(x,z,l):z\in l\in JB(X,Y),x\in l\cap X\}$ and similarly for $N(Y)$.
Let $W(X)=\overline{pr}_{12}(N(X))$ and $W(Y)=\overline{pr}_{12}(N(Y))$. These play the role of Dale's $S(X,i)$.

It follows that $W(X)$ and $W(Y)$ are reduced and irreducible of dimension $\dim(X)+\dim(Y)+1$. These also come with set-theoretic descriptions, 
$$W(X)=\{(x,z):x\in X, z\in L\in JL(X,Y)\}\hspace{.25in} \mathrm{and}\hspace{.25in}W(Y)=\{(z,y):y\in Y, z\in L\in JL(X,Y)\}.$$
We also have $\EJ(X,Y)=pr_2(W(X))=pr_1(W(Y))$. It is also not hard to see that $pr_{23}(N(X))=JB(X,Y)=pr_{13}(N(Y))$.

Consider the projections
$$\begin{tikzcd}[column sep=small]
    &\PP^n\times\PP^n\times\PP^n\times G\arrow[dl,"pr_{124}"]\arrow[dr,"pr_{234}"]\\
    \PP^n\times\PP^n\times G&&\PP^n\times\PP^n\times G
\end{tikzcd}$$
Let $M(X,Y)=pr_{124}^{-1}(N(X))\cap pr_{234}^{-1}(N(Y))$; $M(X,Y)$ is the closure of the set $\{(x,z,y,l):z\in l\in JB(X,Y),x\in l\cap X, y\in l\cap Y\}$, it is irreducible, and $\dim(M(X,Y))=\dim(X)+\dim(Y)+1$. We have a rational map $\begin{tikzcd}\Phi:M(X,Y)\arrow[r,dashed]&\RJ(X,Y)\end{tikzcd}$ given by $\Phi(x,z,y,l)=a[x,0]+b[0,y]$, where we treat $l$ as $\PP^1$ with homogeneous coordinates $[a,b]$, with $x$ at $[1,0]$ and $y$ at $[0,1]$.

This gives us the commutative diagram
\begin{equation}
    \begin{tikzcd}\label{Dale1.7}
        &M(X,Y)\arrow[dr,"\overline{pr}_{124}"]\arrow[dl,swap,"\overline{pr}_{234}"]\arrow[ddrr,bend left=30,"\Phi"]
        \\
        N(X)
        \arrow[dd,"\overline{pr}_{12}",swap]
        \arrow[dr,"\overline{pr}_{23}"]&&
        N(Y)\arrow[dl,swap,"\overline{pr}_{13}"]
        \arrow[dd,"\overline{pr}_{12}"]
        \\
        &JB(X,Y)\arrow[dd,"\beta"]&&\RJ(X,Y)\arrow[ddll,bend left=30,"\pi"]\\
        W(X)\arrow[dr,swap,"\alpha_X"]&&
        W(Y)\arrow[dl,"\alpha_Y"]\\
        &\EJ(X,Y)
    \end{tikzcd}
\end{equation}
The two squares which share the morphism $\beta$ will behave almost identically to Dale's diagram (1.7). The remaining square is in fact a fiber square (this can be gleaned from the set theoretic descriptions of $N(X)$ and $N(Y)$); we will leverage that fact in \S3. The map $\overline{pr}_{12}$ induces isomorphisms between open dense subsets of $N(X)$ and $W(X)$, and of $N(Y)$ and $W(Y)$. The map $\Phi$ is also birational.
\subsection{The general join line of $\EJ(X,Y)$}
The general secant line to a variety may meet the variety more than twice. The following definition is the ``correct" generalization of this concept.
\begin{definition}
    We say that a join line is an $(m_X,m_Y)$-join line of $\EJ(X,Y)$ if $\length(X\cap L)=m_X$ and $\length(Y\cap L)=m_Y$. We say that the pair $X,Y\subseteq\PP^n$ are $(m_X,m_Y)$-joined if a general join line is an $(m_X,m_Y)$-join line.
\end{definition}
\begin{figure}[h]
    \centering
    \includegraphics[width=.8\textwidth]{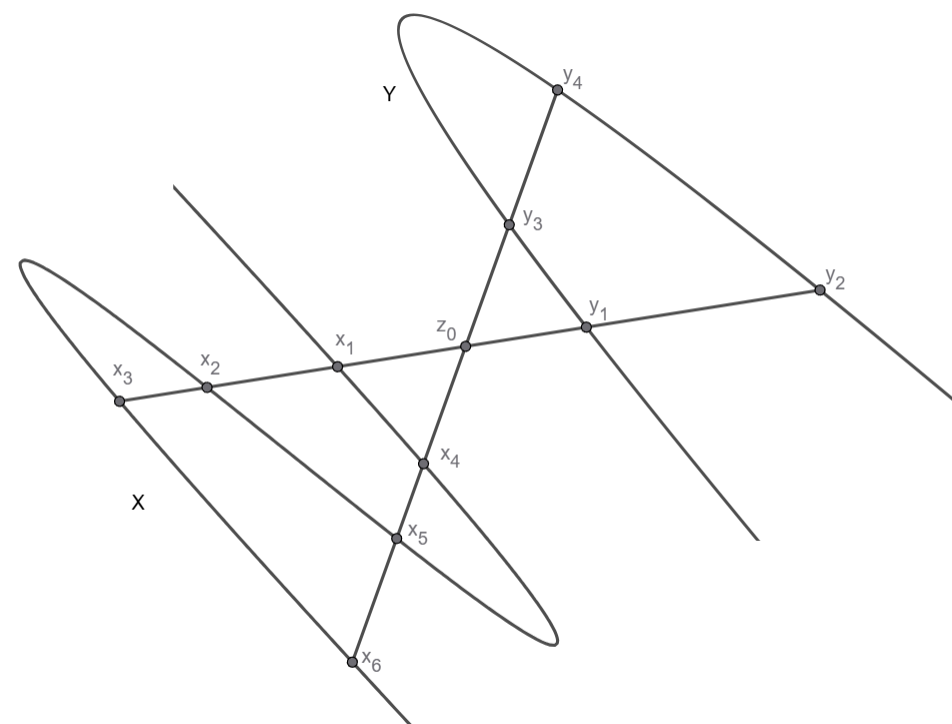}
    \caption{Sketch of curves $X$ and $Y$ with a pair of (3,2)-join lines intersecting at $z_0$.}
    \label{figure}
\end{figure}
It is straightforward to show that there is a pair of positive integers $(m_X,m_Y)$ such that general join lines are $(m_X,m_Y)$-join lines, and this fact was first stated by Ballico (\cite{EBallico}). It is also easy to see that $\EJ(X,Y)$ is (1,$m_Y$)-joined if and only if $X\neq\Sec Y$, and symmetrically for $m_Y$.

\begin{example}
    Assume that $\EJ(X,Y)$ is (3,2)-joined, and that general points of $\EJ(X,Y)$ lie on two distinct join lines. (This situation is sketched in figure \ref{figure}.) We will see later that this makes $\deg(\alpha_X)=3\cdot2=6$, $\deg(\alpha_Y)=2\cdot2=4$, and $\deg(\beta)=2$.
    
    Recalling the morphism $\begin{tikzcd}\pi:\RJ(X,Y)\arrow[r,dashed]&\EJ(X,Y)\end{tikzcd}$, it is important to note that $\pi^{-1}(z_0)$ contains 12 points, one preimage on each of the 6 lines $\overline{[x_1,0],[0,y_1]}$, $\overline{[x_2,0],[0,y_1]}$, $\overline{[x_3,0],[0,y_1]}$, $\overline{[x_1,0],[0,y_2]}$, $\overline{[x_2,0],[0,y_2]}$, $\overline{[x_3,0],[0,y_2]}\subseteq\RJ(X,Y)\subseteq\PP^{2n+1}$, and similarly for the other join line. This demonstrates that $\deg(\pi)\neq\left(m_X+m_Y\right)\deg(\beta)$, contrary to what one might expect.
\end{example}
\subsection*{1.3 Strange and constrained varieties}
A nondegenerate subvariety $X\subseteq\PP^n$ is called \textit{strange} for a linear space $L\subseteq\PP^n$ if for all $x\in X$, the embedded tangent space of $X$ at $x$, which we will write $T_{X,x}$, intersect $L$. $X$ is called \textit{strongly strange} for $L$ if each tangent space contains $L$. The generalizations we will show require the following 
\begin{definition}
    We call a pair $X,Y\subseteq\PP^n$ a (strongly) $L$-strange pair, or simply a strange pair, if all tangent spaces of $X$ and of $Y$ contain $L$.
\end{definition}
This definition allows both members of a non-strange pair to be individually strange, as long as they are not strange for the same $L$. We also recover the definition of strange: $X$ is strange if and only if $X,X$ is a strange pair.
% \begin{example} Consider the curves $X=V(x^p-yu^{p-1},y^p-zu^{p-1},z^p-wv^{p-1})$, $Y=V(x^p-zu^{p-1},z^p-yu^{p-1},y^p-wv^{p-1})$, and $Z=V(z^p-yu^{p-1},y^p-xu^{p-1},x^p-wv^{p-1})$ in $\PP^4$. Each is strange through a point, $V(v,y,z,w)$ for $X$ and $Y$ and $V(v,y,x,w)$ for $Z$. $X,Y$ is a strange pair, while $X,Z$ is not.
%%Refer back to this example later with a calculation of the degree of $\pi$ (since its inseparable).
% \end{example}

We can define a \textit{constrained pair} in a similar fashion. Let
$$t(X,Y)=\min\{\dim(T_{X,x}\cap T_{Y,y}):(x,y)\in X\times Y\}$$
There is an open dense subset $U$ of $X\times Y$ where $\dim(T_{X,x}\cap T_{Y,y})=t(X,Y)$ for all $(x,y)\in U$. By Terracini's lemma (our Lemma \ref{terracini} below),
$$\dim(X)+\dim(Y)-t(X,Y)=\dim\span(T_{X,x},T_{Y,y})\leq\dim\EJ(X,Y)$$
We say $X,Y$ are a constrained pair if $\dim(\EJ(X,Y))>\dim(X)+\dim(Y)-t(X,Y)$.

\cite{Dale} records some elementary facts about $t(X)$; we get fewer interesting elementary facts about $t(X,Y)$. We have a version of Dale's elementary fact (a): $t(X,Y)\leq \min\{\dim(X),\dim(Y)\}$, and $t(X,Y)=\min\{\dim(X),\dim(Y)\}$ if and only if $X$ is strongly strange for $L$ and $Y\subseteq L$, or vice versa. However, we do not get a satisfactory version of Dale's elementary fact (b), and the rest of the properties Dale lists are contingent on a statement like (b) limiting the size of $t(X,Y)$. It seems that we get fewer of these properties due to the greater variability that a second variety $Y$ provides.

% \begin{proposition}
%     Let $X,Y\subset\PP^n$ varieties. Suppose $X\cup Y$ is nondegenerate, $X$ a curve, $Y\neq\PP^n$. Then $X,Y$ is a constrained pair if and only if they are a strange pair.
% \end{proposition}
% \begin{proof}
%     Assuming $X\cup Y$ is nondegenerate guarantees that $\dim(\EJ(X,Y))=\min\{n,\dim(X)+\dim(Y)+1\}$.
%     That (a) implies that $X,Y$ are a constrained pair comes from the fact that $\dim(T_{X,x}\cap T_{Y,Y})\geq $ first case
% \end{proof}
% If $X\cup Y$ is degenerate, then we can apply this proposition to the linear span $\span(X\cup Y)\cong\PP^m$. Thus, an unconstrained pair is simply a pair whose dimensions aren't too big and which do not form a strange pair.
\section{Generalizing Dale's Results}
\subsection{Generalities on Joins}
The following was first stated in \cite{EBallico}. No proof was given, as the difference can be overcome on intuition alone, but we will sketch the idea nonetheless.
\begin{lemma}
    Assume that $\EJ(X,Y)$ is $(m_X,m_Y)$-joined. Then there exists an open, dense subset $U\subset JL(X,Y)$ such that for any $l\in U$, $l$ intersects $X$ transversely in exactly $m_X$ points, and $l$ intersects $Y$ transversely in exactly $m_Y$ points.
\end{lemma}
\begin{proof}
    The aforementioned intuition is that we just need to swap out the correct versions of Dale's ``bad" secants; indeed, the proof is identical once we do so.

    Consider the diagram $\begin{tikzcd}[column sep=small]
        X\times Y&\arrow[l]\Gamma(X,Y)\arrow[r]&JL(X,Y)
    \end{tikzcd}$ and consider the ``bad" join lines:
    \begin{enumerate}[label=(\roman*)]
        \item $\overline{pr}_3(\overline{pr}_{12}^{-1}(\mathrm{Sing}(X)\times Y\cup X\times\mathrm{Sing}(Y)))=\{$join lines intersecting Sing$(X)\cup$Sing$(Y)\}$;
        \item $JL(X,Y)\cap( T(X)\cup T(Y))=\{$ join lines which are tangent to either $X$ or $Y\}$ ($T(X)\subset G$ is the set of tangent lines to $X$);
        \item $JL_{>(m_X,m_Y)}(X,Y)=\{$ join lines $l$ for which either $\length(l\cap X)>m_x$ or $\length(l\cap Y)>m_Y\}$;
        \item $JL_\infty(X,Y)=\{$join lines which are contained in $X\cup Y\}$.
    \end{enumerate}
    Each of these is proper and closed in $JL(X,Y)$, so we can take $U$ to be their complement.
\end{proof}
\subsection{Terracini's Lemma}
\AA dlandsvik (\cite{dlandsvik1987JoinsAH}) has previously extended Terracini's lemma to embedded joins, alongside many other authors in various generalities (for instance \cite{Terracinipaper} and \cite{terracinipaper2}).
\begin{lemma}(Terracini's Lemma)
    \begin{enumerate}[label=(\roman*)]\label{terracini}
        \item Let $X\subset\PP^n$ be a reduced and irreducible projective variety, possibly singular, over an algebraically closed field $k$. Let $x,y$ be distinct points of $X,Y$, respectively, and let $z$ be any point on the join line $\overline{xy}$. Then
        \begin{equation}
            \span(T_{X,x},T_{Y,y})\subseteq T_{\EJ(X,Y),z}\label{Terracini}
            \end{equation}  
        \item If $\mathrm{char}(k)=0$, there is an open, dense subset $U\subset\EJ(X,Y)$ such that for any $z\in U$ and for any $(x,y)\in (X\times Y-\Delta)$ for which $\overline{xy}$ passes through $z$, equality holds in \ref{Terracini}.
    \end{enumerate}
\end{lemma}
The proof is identical to Dale's after making appropriate replacements, so we neglect to rewrite it except to clarify one construction which we will use in the next subsection.
%Let $J=\{(x,y,z):x\in X,y\in Y,x\neq y,z\in\overline{xy}\}\subset X\times Y\times\PP^n$. $\overline{pr}_3:J\to\EJ(X,Y)$ is dominant.
Let $\psi:\mathbb{A}^n\times \mathbb{A}^n\times\mathbb{A}^1\to \mathbb{A}^n\times\mathbb{A}^n\times\mathbb{A}^n$ be the morphism given by $\psi(x,y,t)=(x,y,(tx_0+(1-t)y_0,\dots,tx_n+(1-t)y_n))$, where $x=(x_1,\dots,x_n)$ and similarly for $y$. Let $\sigma=pr_3\circ\psi:\mathbb{A}^n\times\mathbb{A}^n\times\mathbb{A}^1\to\mathbb{A}^n$.
\subsection*{2.3 Inseparability}
Let $\begin{tikzcd} \varphi_X=(pr_1,\sigma|_{X\times Y\times\mathbb{A}^1}):X\times Y\times\mathbb{A}^1\arrow[r,dashed]& W(X)\end{tikzcd}$ such that $(x,y,t)\mapsto (x,tx+(1-t)y)$, and define $\begin{tikzcd}\varphi_Y=(\sigma|_{X\times Y\times\mathbb{A}^1},pr_2):X\times Y\times\mathbb{A}^1\arrow[r,dashed]& W(Y)\end{tikzcd}$ similarly. These maps are rational and are only defined on the intersection of $X$ and $Y$ with an appropriate affine patch of $\PP^n$. Both $\varphi_X$ and $\varphi_Y$ dominate their targets.
\begin{lemma}
    The morphisms $\varphi_X$ and $\varphi_Y$ are generically unramified.
\end{lemma}
The proof of this lemma is again identical to Dale's after we make the appropriate substitutions. 

Continuing in the same vein, we can build a subset $G(X)\subset W(X)\subset\PP^n\times\PP^n$ as follows: $(x,z)\in G$ if and only if
\begin{enumerate}[label=(\alph*)]
    \item $(x,z)\in \mathrm{Reg}(W(X))$
    \item $x\neq z$
    \item there exists $y\in\overline{x,z}\cap Y$ with $x\notin T_{Y,y}$
\end{enumerate}
Then $G(X)$ is a constructible, dense subset of $W(X)$. We similarly construct $G(Y)$ in $W(Y)$.
\begin{proposition}
    Let $(x,z)\in G(X)$ ad let $y\in (\overline{zx}\cap Y)$ be any point such that $x\notin T_{Y,y}$. Then $x\in\mathrm{Reg}(X)$ and $y\in\mathrm{Reg}(Y)$, and
    \begin{enumerate}[label=(\roman*)]
        \item $pr_1(T_{W(X),(x,z)})=T_{X,x}$
        \item $T_{W(X),(x,z)}\cap\{x\}\times\PP^n=\{x\}\times\span(T_{Y,y},x)$
    \end{enumerate}
    We have assumed that $x\neq z$ and $x\neq y$. If, in addition, $z\neq y$, we also have
    \begin{enumerate}[label=(\roman*),resume]
        \item $pr_1(T_{W(X),(x,z)})=\span(T_{X,x},T_{Y,y})$
        \item $T_{W(X),(x,z)}\cap\PP^n\times\{z\}=\span(T_{X,x}\cap T_{Y,y},x)\times\{z\}$
    \end{enumerate}
    And symmetrically for $G(Y)$.
\end{proposition}
Once again, the proof is identical after making appropriate notational changes; this proof is where we needed $\varphi_X$ and $\varphi_Y$ to be generically unramified.
\begin{corollary}\label{Dale3.4}
    The morphism $\alpha_X:W(X)\to \EJ(X,Y)$ is separably generated if and only if $X$ and $Y$ are not a constrained pair. Symmetrically, the same is true for $\alpha_Y:W(Y)\to\EJ(X,Y)$.
\end{corollary}
The proof is identical to Dale's after substituting our notation; just as this proof motivated Dale's use of constrained varieties (his only works if the tangent spaces of generic points of $X$ do not meet, i.e. $X$ is unconstrained), the modified proof under our generalization to embedded joins naturally motivates the definition of a constrained pair.
\begin{lemma}\label{Dale3.5}
    Assume that $\EJ(X,Y)$ is $(m_X,m_Y)$-joined. Then the morphism 
    $$\overline{pr}_{13}:N(Y)\to JB(X,Y)$$
    is separable of degree $m_Y$, and the morphism
    $$\overline{pr}_{23}:N(X)\to JB(X,Y)$$
    is separable of degree $m_X$.
\end{lemma}
\begin{proof}
    Again, the same as Dale's proof after changing notation.
\end{proof}
By applying our Corollary \ref{Dale3.4} and Lemma \ref{Dale3.5} to the lower two squares of diagram \ref{Dale1.7}, we get
\begin{theorem}
    \begin{enumerate}[label=(\roman*)]
        \item The morphisms $\alpha_X$ and $\alpha_Y$ are separably generated if and only if $\beta$ is separably generated, and this is so if and only if $X,Y$ is not a constrained pair.
        \item Assume that $\EJ(X,Y)$ is $(m_X,m_Y)$-joined, and assume that $\dim(\EJ(X,Y))=\dim(X)+\dim(Y)+1$. Then
        $$\deg\alpha_X=m_X\deg\beta\hspace{.5in}\mathrm{and}\hspace{.5in}\deg\alpha_Y=m_Y\deg\beta$$
    \end{enumerate}
\end{theorem}
\subsection{Degree of the Embedded Join of a Curve}
A version of the following theorem is recorded as proposition 2.5 in \cite{Achilles}. Those authors examine curves over $\mathbb{C}$, where the only strange varieties are linear space, hence they do not mention strangeness. Because we can have varieties $X$ and $Y$ of different dimensions, Dale's theorem 4.1 generalizes much more broadly than the results recorded above. The proof follows Dale's closely, until we must contend with possibilities arising from allowing $Y$ to have greater than 1 dimension.

\begin{theorem}\label{Dale4.1}
    Let $X,Y\subset\PP^n$ be reduced and irreducible, assume $X$ is a nondegenerate curve, $\dim(Y)\leq n-3$, $\EJ(X,Y)$ is $(m_Y,1)$-joined, and either
    \begin{enumerate}[label=(\alph*)]
        \item $Y$ is not the closure of a union of any set of lines, or 
        \item $F_1(Y)=\{L\subset Y\}\subset G(1,n)$ has strictly smaller dimension than $Y$.
    \end{enumerate}
    (In particular, (b) is satisfied if $\dim(Y)\leq 3$ and $Y$ is not a 1-parameter family of 2-planes.) Then $\beta:JB(X,Y)\to\EJ(X,Y)$ is birational if and only if $X,Y$ are not a strange pair.
\end{theorem}
\begin{proof}
    Under these circumstances, $X,Y$ are not a strange pair if and only if they are not a constrained pair, if and only if $\beta$ is separable. If $\deg(\beta)=1$, it must be separable and therefore $X,Y$ are not a strange pair. Now supposing that $X,Y$ are not a strange pair, it is enough to show that the number of points in the general fiber of $\beta$ is 1. This is equivalent to the following lemma.
\end{proof}
\begin{lemma}\label{Dale4.2}
    Let $X,Y\subset\PP^n$ as above, and assume $X,Y$ are not a strange pair. Then there is an open, dense subset $U\subset\EJ(X,Y)$ such that through any point $z\in U$ there passes exactly one join line of $\EJ(X,Y)$.
\end{lemma}
\begin{proof}
    There exists an open, dense subset $U\subseteq\mathrm{Reg}(\EJ(X,Y))$ such that the following hold.
    \begin{enumerate}[label=(\alph*)]
        \item Through any $z\in U$ there passes a fixed number $\deg(\beta)\geq1$ of join lines of $\EJ(X,Y)$. (Apply EGA IV 9.7.8 to $\beta$.)
        \item Any join line through $z\in U$ intersects $X$ and $Y$ in exactly $m_X$ and $m_Y$ regular points, respectively.
        \item Through any $z\in U$ there passes at least one join line $L=\overline{xy}$ such that $T_{X,x}\cap T_{Y,y}=\emptyset$.
    \end{enumerate}
    Now assume that $\deg(\beta)\geq2$. Let $z_0\in U$. We know that $z_0$ lies on a secant $L_0=\overline{x_0,y_0}$ such that the dimension of $H_0=\span(T_{X,x_0},T_{Y,y_0})$ is $\dim(Y)+2\leq n-1$. Terracini's lemma says $H_0=T_{\EJ(X,Y),z}$ for any $z\in L_0\cap U$. Picking general $z\neq z_0$ in $L_0$, there must be a join line $L=\overline{x_1,y}\ni z$, and Terracini now says that $\span(T_{X,x_1},T_{Y,y})\subseteq H_0$. In particular, $x_1\neq x_0$. 
    
    Two things can now happen. As we vary $z$ along $L_0\cap U$, $\overline{x,y}$ also varies. In this case, we get $x\in\span(T_{X,x},T_{Y,y})\subseteq H_0 $ for infinitely many points $x\in X$; it follows that $X\subset H_0$, contradicting our assumption that $X$ is nondegenerate. Therefore, we may assume the other possibility: as $z$ varies along $L_0\cap U$, we have $\overline{x,y}\ni x_1$; this happens if $\overline{x_0,x_1,y_0}\cap Y$ has infinitely many points, in which case it will be a plane curve. These plane curves will generally have the fixed degree $m_Y=1$. Then for general $y\in Y$, we can vary $x_0$, and we get a 1-parameter family of lines through $y$ contained in $X$; a simple incidence correspondence argument shows that $\dim(F_1(Y))\geq\dim(Y)$, but also $\dim(F_1(X))\geq2$. This contradicts assumptions (a) or (b), so we can conclude that $\deg(\beta)=1$.

    We now consider the ``in particular" statement. Since $\dim(F_1(Y))\geq2$, we may assume $Y$ is not a curve. We have the classical result that $\dim(F_1(Y))=2\dim(Y)-2$ if and only if $Y$ is a linear space, and otherwise $\dim(F_1(Y))\leq 2\dim(Y)-3$. Beniamino Segre's work in \cite{Segre} shows that $\dim(F_1(Y))=2\dim(Y)-3$ if and only if $Y$ is a quadric hypersurface or $Y$ is a 1-parameter family of 2-planes. We have assumed that $Y$ is neither, so $$\dim(F_1(Y))\leq2\dim(Y)-4\leq\dim(Y)-1<\dim(Y)$$
    This guarantees assumption (b) is satisfied.
\end{proof}
\begin{remark}
    The argument in the last paragraph of the above proof could be extended using further results on the shape of $X$ when $\dim(F_1(X))=2\dim(X)-2-N$ for larger $N$, such as those found in \cite{Rogora}.

    The entire proof can be extended to the case where $\mathrm{EJ}(X,Y)$ is $(m_X,m_Y)$-joined with $m_Y>1$, by replacing criteria (a) and (b) with corresponding statements about plane curves of degree $m_Y$.
\end{remark}
These assumptions provide criteria for testing whether Theorem \ref{Dale4.1} could be applied to a given $Y$. General varieties pass both criteria. Cones over general varieties fail criterion (a) but pass criterion (b). A rational normal $k$-fold scroll with $k\geq3$ fails both criteria. 

Ballico has a version of Theorem \ref{Dale4.1} with a much simpler set of hypotheses (Theorem 1 in \cite{EBallico}); this seems to come from the fact those hypotheses are equivalent to $\deg(\beta)=1$. We can slightly extend Ballico's work by observing that we only need the nondegeneracy of $X$ and $Y$ to conclude that $\dim(\EJ(X,Y))=\dim(X)+\dim(Y)+1$; but this will happen if and only if $X\cup Y$ is nondegenerate, so we don't need them to be individually nondegenerate.
\begin{proposition}
    Assume $X,Y\subset\PP^n$, $X$ a curve, $\dim(Y)\leq n-3$, and $X\cup Y$ nondegenerate. Then general points in $\EJ(X,Y)$ are contained in a unique join line if and only if for general $(a,b,c)\in X\times X\times Y$, $\overline{a,b,c}\cap Y=\{c\}$ (set theoretically).
\end{proposition}
\begin{proof}
    ($\impliedby$) See \cite{EBallico}.
    
    ($\implies$) Pick general $a,b\in X$ and $c\in Y$. For general $z\in \overline{a,c}$, we must have $\overline{z,c}\cap Y=\emptyset$, and hence $\overline{a,b,c}\cap Y=\{c\}$.
\end{proof}
The content of Dale's Theorem 4.3, generalized to embedded joins, is contained in Theorem 2.4 of \cite{dlandsvik1987JoinsAH}.
\section{Degree of $\pi:\RJ(X,Y)\dashrightarrow\EJ(X,Y)$}
As mentioned in \S1, the following sub-diagram of diagram \ref{Dale1.7} is a fiber square.
$$\begin{tikzcd}
    &M(X,Y)\arrow[dr,"\overline{pr}_{124}"]\arrow[dl,swap,"\overline{pr}_{234}"]\\
    N(X)\arrow[dr,"\overline{pr}_{23}"]&&
    N(Y)\arrow[dl,swap,"\overline{pr}_{13}"]\\
    &JB(X,Y)
\end{tikzcd}$$
This fact immediately gives the 
\begin{proposition}
    The morphism $\overline{pr}_{234}:M(X,Y)\to N(X)$ is separable of degree $m_Y$ and the morphism $\overline{pr}_{124}:M(X,Y)\to N(Y)$ is separable of degree $m_X$. 
\end{proposition}
Collecting all of this information together, we get our main theorem.
\begin{theorem}\label{main}
    Let $X,Y\subset\PP^n$ be closed subschemes. 
    \begin{enumerate}[label=(\alph*)]
        \item The following are equivalent
        \begin{enumerate}[label=(\roman*)]
            \item The morphism $\begin{tikzcd}\pi:\RJ(X,Y)\arrow[r,dashed]&\EJ(X,Y)\end{tikzcd}$ is separably generated.
            \item The morphism $\alpha_X:W(X)\to\mathrm{EJ}(X,Y)$ is separably generated.
            \item The morphism $\alpha_Y:W(Y)\to\mathrm{EJ}(X,Y)$ is separably generated.
            \item The morphism $\beta:JB(X,Y)\to\mathrm{EJ}(X,Y)$ is separably generated.
            \item $X,Y$ is not a constrained pair.
        \end{enumerate}
        \item Assume that $\EJ(X,Y)$ is $(m_X,m_Y)$-joined, and assume that $\dim(\EJ(X,Y))=\dim(X)+\dim(Y)+1$. Then
        $$\deg(\pi)=m_Xm_Y\deg(\beta)$$
    \end{enumerate}
\end{theorem}
\section{Further Questions}
It is clear that by generalizing from secant varieties to embedded joins, some things become slightly nicer (e.g. $X,Y$ can both be strange, as long as they are not strange `in the same way') and some things become less nice (e.g. we do not recover all of Dale's elementary properties of constrained varieties). It is not obvious how this will affect the complexity of the answers to the following questions.

It would be nice to know when the different factors in $\deg(\pi)=m_Xm_Y\deg(\beta)$ are equal to 1, or if they can otherwise be bounded from below. As mentioned above, it is clear that $m_X=1$ if and only if $Y\not\subseteq\Sec(X)$, and similarly for $m_Y$; is there a simple way to test this without explicitly calculating the secant varieties? Determining precisely when $\beta:SB(X,Y)\to\EJ(X,Y)$ is birational, even just for a curve and a non-curve, may not have a nice answer; Ballico's version of our Theorem \ref{Dale4.1} (Theorem 1 in \cite{EBallico}) shows a direction one could move. Dale insists that even for very nice things, his lemma 4.2 (and presumably also the present paper's Lemma \ref{Dale4.2}) does not work (see the remark preceding theorem 4.3 in \cite{Dale}).

It could be interesting to find versions of Dale's elementary properties of constrained varieties which hold for constrained pairs.

Ballico has extended Dale's results to $\EJ(X,X,\dots,X)$, where $X$ is repeated $k$ times, called either the $k$-secant or $(k-1)$-secant variety of $X$ (\cite{Ballico2}). It may be possible to extend a fusion of the above results and those of Ballico to the embedded join of more than two varieties, $\EJ(X_1,X_2,\dots,X_r)$, after appropriate definitions of ``non-strange $r$-tuple" and modifications to the constructions from \S1, though the usefulness of such work is dubious at the moment.

%tuesday night: 1:35
%wednesday library: 1:25

%Additionally, there may be more literature on secant varieties which is interesting, useful, and immediately provable in the situation of embedded joins of 2 or more varieties.

%People are saying this paper also classified all ``defective" surfaces? Not sure what that is, or where it is in the paper. They probably meant one of Dale's other papers, with the severi theorem for veronese surfaces or the secants of algebraic surfaces paper. 

\printbibliography
\noindent\textit{Email address}: jjbeckma@syr.edu\\
Department of Mathematics, Syracuse University, Syracuse NY 13244, USA
\end{document}